\documentclass[11pt,twoside,a4paper]{amsart}
\usepackage[utf8]{inputenc}
\usepackage{ dsfont }
\usepackage{amsmath,amssymb}
\usepackage{mathrsfs}
\usepackage{amsfonts}
\usepackage{extarrows}
\usepackage{mathtools}
\usepackage{microtype}
\usepackage{amsthm}
\usepackage{tikz-cd}
\usepackage[english]{babel}
\usepackage{xcolor}
\usepackage{graphicx}
\usepackage{url}
\usepackage[hidelinks]{hyperref}
\newcommand{\p}{\mathbb{P}}
\newcommand{\Z}{\mathbb{Z}}
\newcommand{\R}{\mathbb{R}}
\newcommand{\Q}{\mathbb{Q}}
\newcommand{\C}{\mathbb{C}}

\newcommand{\KS}{\mathrm{KS}}
\newcommand{\dRB}{\mathrm{dRB}}
\newcommand{\Cl}{\mathrm{Cl}}
\newcommand{\la}{\lambda_{\psi}}

\theoremstyle{plain} 
\newtheorem{thm}{Theorem}[section] 

\newtheorem{cor}[thm]{Corollary} 
\newtheorem{lem}[thm]{Lemma} 
\newtheorem{prop}[thm]{Proposition}
\newtheorem{conj}[thm]{Conjecture} 
\theoremstyle{definition} 
\newtheorem{defn}[thm]{Definition}
\newtheorem{exmp}[thm]{Example}
\theoremstyle{remark} 
\newtheorem{rmk}[thm]{Remark}

\begin{document}

\title[Hodge similarities, algebraic classes, and Kuga--Satake varieties]{Hodge similarities, algebraic classes, and Kuga-Satake varieties}  
\author[M. Varesco]{Mauro Varesco}
\address{Mathematisches Institut, Universit\"at Bonn, Endenicher Allee 60, 53115 Bonn, Germany}
\email{varesco@math.uni-bonn.de}

\begin{abstract}
We introduce in this paper the notion of Hodge similarities of transcendental lattices of hyperk\"ahler manifolds and 
investigate the Hodge conjecture for these Hodge morphisms. Studying K3 surfaces with a symplectic automorphism, we prove the Hodge conjecture for the square of the general member of the first four-dimensional families of K3 surfaces with totally real multiplication of degree two.
We then show the functoriality of the Kuga--Satake construction with respect to Hodge similarities. This implies that, if the Kuga--Satake Hodge conjecture holds for two hyperk\"ahler manifolds, then every Hodge similarity between their transcendental lattices is algebraic after composing it with the Lefschetz isomorphism. In particular, we deduce that Hodge similarities of transcendental lattices of hyperk\"ahler manifolds of generalized Kummer deformation type are algebraic.
\end{abstract}
\maketitle

\section*{Introduction}

\subsection{Hyperk\"ahler manifolds and the Hodge conjecture.}
Let $X$ be a hyperk\"ahler manifold, and let $T(X)\subseteq H^2(X,\Q)$ be its \textit{transcendental lattice}, which is the orthogonal complement of the N\'eron--Severi group of $X$ in $H^2(X,\Q)$ with respect to the Beauville--Bogomolov quadratic form.
The relevance of this notion in the context of the Hodge conjecture can be evinced from the following observation: let $X$ and $Y$ be hyperk\"ahler manifolds. By Lefschetz $(1,1)$ theorem, a Hodge morphism $H^2(X,\Q)\rightarrow H^2(Y,\Q)$ is algebraic if and only if the induced Hodge morphism $T(X)\rightarrow T(Y)$ is algebraic. Recall that a Hodge morphism $H^2(X,\Q)\rightarrow H^2(Y,\Q)$ is said to be algebraic if the corresponding Hodge class in $H^{2n,2n}(X\times Y,\Q)$ is algebraic, where $2n$ is the dimension of $X$.

\smallskip
In general, it is not known whether Hodge morphisms of transcendental lattices are algebraic or not. However, there have been promising results for the class of Hodge isometries. Recall that a Hodge isomorphism $T(X)\rightarrow T(Y)$ is called a Hodge isometry if it is an isometry with respect to the Beauville--Bogomolov quadratic forms on $X$ and $Y$. A result by Buskin \cite{buskin2019every} reproved by Huybrechts \cite{huybrechts2019motives} shows that Hodge isometries of transcendental lattices of projective K3 surfaces are algebraic. The same has been proven by Markman \cite{markman2022rational} for Hodge isometries of transcendental lattices of hyperk\"ahler manifolds of $\textrm{K3}^{[n]}$-type.

\smallskip
In this paper, we introduce a natural generalization of Hodge isometries which we call \textit{Hodge similarities}: a Hodge isomorphism is a Hodge similarity if it multiplies the quadratic form by a non-zero scalar called \textit{multiplier}, see Definition \ref{def. Hodge similarity}.
Note that Hodge isometries are Hodge similarities with multiplier one.
There are two contexts where Hodge similarities naturally appear. The main instance is given by hyperk\"ahler manifolds $X$ whose
endomorphism field $E\coloneqq\mathrm{End}_{\mathrm{Hdg}}(T(X))$ is a totally real field of degree two: indeed, every totally real field of degree two is isomorphic to $\Q(\sqrt{d})$ for some positive integer $d$. One then sees that $\sqrt{d}\colon T(X)\rightarrow T(X)$ is a Hodge similarity. This follows immediately from the fact that, as $E$ is totally real, the Rosati involution is the identity. Note that in this case $E$ is generated by Hodge similarities.
A second source of examples of Hodge similarities is the following: given a hyperk\"ahler manifold $X$, there might exist another hyperk\"ahler manifold $Y$ with transcendental lattice Hodge isometric to $T(X)(\lambda)$, for some $\lambda\in \Q_{>0}$, where $(\lambda)$ indicates that the quadratic form is multiplied by $\lambda$. The identity of $T(X)$ then defines a natural Hodge morphism $T(Y)\rightarrow T(X)$ which is a Hodge similarity. 
At the time of writing this paper, there are very few examples of Hodge similarities that are not isometries which can be proven to be algebraic. For example, in the case of K3 surfaces with totally real endomorphism field $E=\Q(\sqrt{d})$, the algebraicity of $\sqrt{d}$ has been proven only for some one-dimensional families of such K3 surfaces. This is a result by Schlickewei \cite{schlickewei2010hodge} which has then been extended in \cite{varesco2022hodge}.  Note that the proof in the references involves the study of the Hodge conjecture for Kuga--Satake variety of these K3 surfaces, and does not use the fact that $E$ is in these cases generated by Hodge similarities.

\subsection{Hodge similarities of K3 surfaces and symplectic automorphisms}
Recall that the Hodge conjecture for the product of two K3 surfaces $X$ and $Y$ to the algebraicity of the elements of $\mathrm{Hom}_{\mathrm{Hdg}}(T(X),T(Y))$. This follows from the K\"unneth decomposition and the fact that the quadratic form  $q_X$ identifies $(T(X)\otimes T(Y))^{2,2}\cap (T(X)\otimes T(Y))$ with $\mathrm{Hom}_{\mathrm{Hdg}}(T(X),T(Y))$. As mentioned above, Hodge isometries between the transcendental lattices of two K3 surfaces are known to be algebraic. In particular, the Hodge conjecture holds for $X\times Y$ whenever $\mathrm{Hom}_{\mathrm{Hdg}}(T(X),T(Y))$ is generated by Hodge isometries. This is the case when $T(X)$ and $T(Y)$ are Hodge isometric and $\mathrm{Hom}_{\mathrm{Hdg}}(T(X),T(Y))$ is $\Q$ or a CM field.

\smallskip
The main result of Section \ref{Sec. K3 with a symplectic autom.} is the proof of the algebraicity of some Hodge similarities for some families of K3 surfaces with totally real multiplication of degree two:
\begin{thm}[Theorem \ref{thm.: algebraicity of sqrt p}, \ref{thm. Main thm for Nik. inv}, and \ref{main thm for order three}]
\label{thm: Intro, algebraic similarities and symplectic}
Let $X$ be a K3 surface Hodge isometric to a K3 surface with a symplectic automorphism of order $p$ with $p=2,3$. Assume furthermore that $\Q(\sqrt{p})$ is contained in the endomorphism field of $X$. Then, $\sqrt{p}\colon T(X)\rightarrow T(X)$ is algebraic.
In particular, the Hodge conjecture for $X\times X$ holds if $\mathrm{End}_{\mathrm{Hdg}}(T(X))\simeq\Q(\sqrt{p}).$
\end{thm}
The condition ``$X$ is Hodge isometric to a K3 surface with a symplectic automorphism of order $p$" is equivalent to $T(X)\hookrightarrow U^3_\Q\oplus E_8(-2)_\Q$ for $p=2$ and to $T(X)\hookrightarrow U^3_\Q\oplus (A_2)^2_\Q$ for $p=3$. This is deduced in Proposition \ref{prop.: criterion for Nikulin involution} and Proposition \ref{prop.: criterion for symp. autom. of order 3} from the classical result by Nikulin \cite{Nikulin79}, van Geemen and Sarti \cite{van2007nikulin}, and Garbagnati and Sarti \cite{garbagnati2007symplectic}.
Using these conditions on the transcendental lattice, we show that the families of K3 surfaces satisfying the hypotheses of Theorem \ref{thm: Intro, algebraic similarities and symplectic} are at most four-dimensional for $p=2$ and two-dimensional for $p=3$. We then produce examples of such maximal-dimensional families in Proposition \ref{prop. 4-dim family of K3 admitting a Nik inv} and Proposition \ref{prop. existence of similarity of multiplier 3}.
In particular, Theorem \ref{thm: Intro, algebraic similarities and symplectic} provides the first four-dimensional families of K3 surfaces with totally real multiplication of degree two for which the Hodge conjecture can be proven for the square of its general member and the first two-dimensional family of K3 surfaces with totally real multiplication of degree two for which the Hodge conjecture can be proven for the square of all its members.

\subsection{Kuga--Satake varieties and Hodge similarities}
In Section \ref{sec. KS varieties}, we prove that the functoriality of the Kuga--Satake construction with respect to Hodge isometries extends to Hodge similarities in the following sense:
\begin{prop}[Proposition 
\ref{prop. similarities and KS}]
\label{prop. Intro similarities an KS}
Let $\psi\colon (V,q)\rightarrow (V',q')$ be a Hodge similarity of polarized Hodge structures of K3-type. Then, there exists an isogeny of abelian varieties $\psi_{\KS}\colon \mathrm{KS}(V)\rightarrow\KS(V')$ making the following diagram commute
\[
\begin{tikzcd}
V \arrow[d, hook] \arrow[r, "\psi"] & V' \arrow[d, hook] \\
H^1(\KS(V),\Q)^{\otimes 2} \arrow[r, "({\psi_{\KS}})_*^{\otimes 2}"]      &  H^1(\KS(V'),\Q)^{\otimes 2}
\end{tikzcd},
\]
where the vertical arrows are the Kuga--Satake correspondence for $V$ and $V'$.
\end{prop}
In Section \ref{subsection dRB similarities}, we exploit the observation that a similarity of quadratic spaces induces an isomorphism of even Clifford algebras to extend the result by Kreutz, Shen, and Vial \cite{vial2023dRBconj} which shows that de Rham--Betti isometries between the second de Rham--Betti cohomology of two hyperk\"ahler manifolds defined over $\overline{\Q}$ are motivated in the sense of Andr\'e. We note in Proposition \ref{prop.: dRB similarities} that the same proof as in the reference can be used to show that de Rham--Betti similarities are motivated.

\smallskip

In Section \ref{Sec. implication of functoriality of KS}, we use the functoriality property of the Kuga--Satake construction proven in Proposition \ref{prop. Intro similarities an KS} to deduce the following:
\begin{thm}[Theorem \ref{main thm}]
\label{main thm Intro}
Let $X'$ and $X$ be two hyperk\"ahler manifolds for which the Kuga--Satake Hodge conjecture holds. 
Then, for every Hodge similarity $\psi\colon T(X')\rightarrow T(X)$, the composition 
\[
T(X')\xlongrightarrow{\psi} T(X)\xrightarrow{h^{2n-2}_X\cup\bullet} H^{4n-2}(X,\Q)
\]
is algebraic, where $2n\coloneqq \dim X$ and $h_X$ is the cohomology class of an ample divisor on $X$.
\end{thm}
By a result of Voisin \cite{voisin2022footnotes} based on previous results by Markman \cite{markman2022monodromy} and O'Grady \cite{O'grady}, the Kuga--Satake Hodge conjecture holds for hyperk\"ahler manifolds of generalized Kummer type. This is the main source of examples of manifolds which satisfy the hypotheses of Theorem \ref{main thm Intro}. As the Lefschetz standard conjecture in degree two for these manifolds is proved by Foster \cite{foster2023lefschetz}, Theorem \ref{main thm Intro} shows that Hodge similarities between the transcendental lattices of two hyperk\"ahler manifolds of generalized Kummer type are algebraic. Using the fact that the endomorphism field of these varieties is always generated by Hodge similarities, we then conclude the following:
\begin{thm}[Theorem \ref{thm. similarities and gen kummer}]
\label{intro Thm similarities and general kummer}
Let $X$ and $X'$ be hyperk\"ahler manifolds of generalized Kummer type such that $T(X)$ and $T(X')$ are Hodge similar. Then, every Hodge morphism between $T(X')$ and $T(X)$ is algebraic.
\end{thm}

Note that, opposed to the case of K3 surfaces and hyperk\"ahler manifolds of K3$^{[n]}$-type, already the algebraicity of Hodge isometries was not known in the case of hyperk\"ahler manifolds of generalized Kummer type. Furthermore, note that Theorem \ref{intro Thm similarities and general kummer} also applies for hyperk\"ahler manifolds of generalized Kummer type of different dimension.

\smallskip
In the case of K3 surfaces, the Lefschetz standard conjecture is trivially true. Hence, if the Kuga--Satake Hodge conjecture holds for two given K3 surfaces, Theorem \ref{main thm Intro} shows that every Hodge similarity between their transcendental lattices is algebraic. In particular, this provides a more direct proof of the Hodge conjecture for the square of the K3 surfaces in the one-dimensional families of K3 surfaces with totally real field of degree two studied in \cite{schlickewei2010hodge, varesco2022hodge} that we mentioned above.

\smallskip
For hyperk\"ahler manifolds of $\mathrm{K3}^{[n]}$-type, the Kuga--Satake Hodge conjecture is known only for certain families: the paper \cite{floccari2022sixfolds} proves this conjecture for countably many four-dimensional families of $\mathrm{K3}^{[3]}$-type hyperk\"ahler manifolds. Recall that, for hyperk\"ahler manifolds of $\mathrm{K3}^{[n]}$-type, the Lefschetz standard conjecture has been proven by Charles and Markman \cite{charles2013standard}. Therefore, we deduce the algebraicity of Hodge similarities for the  hyperk\"ahler manifolds of $\mathrm{K3}^{[3]}$-type appearing in \cite{floccari2022sixfolds}.

\smallskip
As a final remark, note that the manifolds $X$ and $X'$ as in Theorem \ref{main thm Intro} are neither assumed to be of the same deformation type nor of the same dimension.

\section*{Acknowledgements}
I would like to thank my PhD.\ supervisor Daniel Huybrechts for the corrections and the numerous suggestions on this article. I am also grateful to Charles Vial for reviewing a preliminary version of this paper and for suggesting me the application of similarities in the context of De Rham--Betti classes. I am furthermore thankful to Alice Garbagnati and Giacomo Mezzedimi for their help on lattice theory and to Eyal Markman for referring me to the article \cite{foster2023lefschetz}.
This research was funded by ERC Synergy Grant HyperK, Grant
agreement ID 854361.

\section{Main definitions}
\label{First definitions}
In this paper, all varieties are assumed to be projective. Unless otherwise stated, the definition field of the varieties we consider is $\C$.

\smallskip
A \textit{hyperk\"ahler manifold} is a simply connected, projective, compact, K\"ahler manifold $X$ such that $H^0(X,\Omega_X^2)$ is generated by a nowhere degenerate symplectic form. Denote by $q_X$ the \textit{Beauville--Bogomolov quadratic form}, which is a non-degenerate quadratic form
on $H^2(X,\Q)$. Recall that $q_X$ induces the following direct sum decomposition
\[
H^2(X,\Q)=\mathrm{NS}(X)_\Q\oplus T(X),
\]
where $\mathrm{NS}(X)$ is the N\'eron--Severi group of $X$ and $T(X)$ is the transcendental lattice of $X$. When talking about the transcendental lattice of a hyperk\"ahler manifold $X$ we will always refer to the rational quadratic subspace $T(X)$ of $H^2(X,\Q)$.
The pair $(T(X),-q_X)$ gives an example of polarized Hodge structures of K3-type:
\begin{defn}
A rational Hodge structure $V$ of weight two is called of \textit{K3-type} if 
\[
\dim_\C V^{2,0}=1, \text{ and } V^{p,q}=0 \text{ for } |p-q|>2.
\]
Moreover, we say that a pair $(V,q)$ is a \textit{polarized Hodge structure of K3-type} if
$q\colon V\otimes V\rightarrow \Q(-2)$ is a morphism of Hodge structures whose
real extension is negative definite on $(V^{2,0}\oplus V^{0,2})\cap V_\R$ and has signature $(\dim V-2,2)$.
\end{defn}
Let $E\coloneqq\mathrm{End}_{\mathrm{Hdg}}(T(X))$ be the endomorphism algebra of the Hodge structure $T(X)$. As $T(X)$ is an irreducible Hodge structure, $E$ is a field. As explained in \cite[Thm.\ 3.3.7]{huybrechts2016lectures}, $E$ is either totally real or CM.
Recall that a field extension $E$ of $\Q$ is totally real if every embedding $E\hookrightarrow \C$ has image contained in $\R$, and it is CM if $E=F(\rho)$, where $F$ is a totally real field and $\rho$ satisfies the following: 
\[
\sigma(\rho)^2\in \sigma(F)\cap \R_{<0},\quad \forall\sigma\colon E\hookrightarrow \C.
\]
These two cases can be distinguished by the action of the \textit{Rosati involution}, which is the involution on $E$ which sends an element $e\in E$ to the element $e'\in E$ such that
\[
q_X(ev,w)=q_X(v,e'w),\quad \forall v,w\in T(X).
\]
The Rosati involution is the identity if $E$ is totally real, and it acts as complex conjugation if $E$ is CM.

\smallskip
As mentioned in the introduction, we focus in this paper on the notion of Hodge similarities:
\begin{defn}
\label{def. Hodge similarity}
Let $(V,q_V)$ and $(V',q_{V'})$ be polarized Hodge structures of K3-type, and let $\psi\colon V\rightarrow V'$ be a Hodge isomorphism. We say that $\psi$ is a \textit{Hodge similarity} if there exists a non-zero $\la\in \Q$ such that
\[
q_{V'}(\psi v,\psi w)=\la q_V(v,w),\quad \forall v,w\in V.
\]
We call $\la$ the \textit{multiplier} of $\psi$.
A \textit{Hodge isometry} is a Hodge similarity $\psi$ of multiplier $\lambda_\psi=1$.
\end{defn}
We say that two hyperk\"ahler manifolds are \textit{Hodge similar} (resp., \textit{Hodge isometric}) if there exists a Hodge similarity (resp., a Hodge isometry) between their transcendental lattices.
Note that the multiplier of a Hodge similarity is always a positive number.

\section{Symplectic automorphisms and algebraic Hodge similarities}
\label{Sec. K3 with a symplectic autom.}
Let $X$ be a K3 surface, and denote by $q$ the polarization on $T(X)$ given by the negative of the intersection form.
Identifying $T(X)$ with its dual via $q$, we see that
\[
\mathrm{End}_{\mathrm{Hdg}}(T(X))\simeq \left(T(X)\otimes T(X)\right)^{2,2} \cap \left(T(X)\otimes T(X)\right).
\]
This shows that proving the Hodge conjecture for $X^2$ is equivalent to showing that every element of $\mathrm{End}_{\mathrm{Hdg}}(T(X))$ is algebraic. In this section, considering K3 surfaces with a symplectic automorphism, we produce examples of K3 surfaces $X$ with $\Q(\sqrt{p})\subseteq\mathrm{End}_{\mathrm{Hdg}}(T(X))$ for which the Hodge similarity $\sqrt{p}$ can be shown to be algebraic.

\smallskip
The starting observation is the following: given a K3 surface $X$ with a symplectic automorphism of order $p$, there exists a K3 surface $Y$ and an algebraic Hodge similarity $\varphi\colon T(Y)\rightarrow T(X)$ of multiplier $p$. To show this, recall that, by \cite[Prop.\ 15.3.11]{huybrechts2016lectures}, the prime $p$ is at most $7$, the fixed locus of $\sigma_p$ is a finite union of points, and the minimal resolution of $X/\sigma_p$ is a K3 surface $Y$. 
Moreover, $Y$ can also be obtained as follows: after a finite sequence of blowups of $X$ at the fixed locus of $\sigma_p$, we get a variety $\widetilde{X}$ with a free action $\widetilde{\sigma}_p$ and $Y\simeq \widetilde{X}/\widetilde{\sigma}_p$. I.e., there is a commutative diagram
\[
\begin{tikzcd}
\widetilde{X} \arrow[d, "\pi"] \arrow[r, "\beta"] & X \arrow[d] \\
Y \arrow[r]  & X/\sigma_p
\end{tikzcd}.
\]
As $\pi\colon \widetilde{X}\rightarrow Y$ is a finite map of degree $p$ and $\beta\colon \widetilde{X}\rightarrow X$ just contracts the exceptional divisors, we see that 
\[
\varphi\coloneqq\beta_*\pi^*\colon T(Y)\rightarrow T(X)
\] is a Hodge similarity of multiplier $p$. Note that $\varphi$ is algebraic.
From this construction, we deduce the following:
\begin{thm}
\label{thm.: algebraicity of sqrt p}
Let $X$ be a K3 surface Hodge isometric to a K3 surface with a symplectic automorphism of prime order $p$. Assume furthermore that  $\Q(\sqrt{p})\subseteq \mathrm{End}_{\mathrm{Hdg}}(T(X))$. Then, the Hodge similarity $\sqrt{p}$ is algebraic.
\begin{proof}
As Hodge isometries of K3 surfaces are algebraic by \cite{buskin2019every} and \cite{huybrechts2019motives}, we may assume that $X$ admits a symplectic automorphism of order $p$. Let $\psi$ be the Hodge similarity of multiplier $p$ on $T(X)$, which exists since $\Q(\sqrt{p})\subseteq \mathrm{End}_{\mathrm{Hdg}}(T(X))$ by assumption.
As remarked above, denoting by $Y$ the minimal resolution of the quotient $X/\sigma_p$, the map $\varphi\coloneqq \beta_*\pi^*\colon T(Y)\rightarrow T(X)$ is a Hodge similarity of multiplier $p$. The composition $\varphi^{-1}\circ \psi\colon T(X)\rightarrow T(Y)$ is then a Hodge isometry. In particular, $\varphi^{-1}\circ \psi$ is algebraic by \cite{buskin2019every} and \cite{huybrechts2019motives}. As $\varphi$ is algebraic, we conclude that $\psi=\varphi\circ (\varphi^{-1}\circ \psi)$ is algebraic. This concludes the proof.
\end{proof}
\end{thm}

\begin{rmk}
\label{rmk. similarity is a squareroot}
The two conditions ``$X$ is isometric to a K3 surface with a symplectic automorphisms of order $p$" and ``the endomorphisms field of $X$ contains $\Q(\sqrt{p})$" are not related. In fact, the general K3 surface with a symplectic automorphism of order $p$ has endomorphism field equal to $\Q$.
Moreover, note that the requirement ``the endomorphisms field of $X$ contains $\Q(\sqrt{p})$" is equivalent to the condition ``$X$ admits a Hodge similarity $\psi$ of multiplier $d$ which is fixed by the Rosati involution". Indeed, if $\psi$ such a Hodge similarity, then $\Q(\psi)$ is a totally real subfield of the endomorphism field of $X$. Using the fact that totally real fields have no non-trivial isometry, we see that $\psi^2/d$ is the identity, i.e., that $\Q(\psi)\simeq \Q(\sqrt{d})$.
\end{rmk}
In the remainder of this section, we construct families of K3 surfaces satisfying the hypotheses of Theorem \ref{thm.: algebraicity of sqrt p}. To do this, we use the following result is adapted from \cite[Sec.\ 3]{van2008real}, we give here a detailed proof for later use.
\begin{prop}
\label{prop. dimension of locus admitting Hodge similarities}
Let $d\in \Z$ be a positive integer which is not a square, and let $(\Lambda,q)$ be a rational quadratic space of signature $(2,\Lambda-2)$ with $\dim \Lambda>4$. Let $\psi$ be a similarity of $\Lambda$ of multiplier $d$ which is fixed by the Rosati involution. Then, $\Lambda$ is even-dimensional, and the locus of Hodge structures of K3-type on $\Lambda$ for which $\psi$ defines a Hodge 
similarity
is either empty or of dimension $(\dim \Lambda)/2-2$.
\begin{proof}
The first statement is immediate from the fact that odd-dimensional quadratic spaces do not admit any similarity of multiplier $d$ if $d$ is not a square.

\smallskip
Let us assume that $\Lambda$ is even-dimensional. As in Remark \ref{rmk. similarity is a squareroot}, we see that, for every Hodge structure on $\Lambda$ for which $\psi$ is a Hodge morphism, $\Q(\psi)\simeq \Q(\sqrt{d})$ is a totally real subfield of the endomorphism field of $\Lambda$. 

\smallskip
Note that $\Lambda$ can be viewed as a $\Q(\psi)$-vector space, that is
$\Lambda\simeq \Q(\psi)^{(\dim \Lambda)/2}$. The decomposition  $\Q(\psi)\otimes_\Q\R\simeq \R_{\sqrt{d}}\oplus\R_{-\sqrt{d}}$ into eigenspaces for the action of $\psi$ then induces a decomposition
\[
\Lambda_\R\simeq \Lambda_{\sqrt{d}}\oplus \Lambda_{-\sqrt{d}},
\]
where $\Lambda_{\sqrt{d}}\coloneqq\{v\in\Lambda_\R\mid  \psi v=\sqrt{d}v\}$ and similarly for $\Lambda_{-\sqrt{d}}$. Note that $\Lambda_{\sqrt{d}}$ and $\Lambda_{-\sqrt{d}}$ are both of dimension $(\dim \Lambda)/2$.
From the fact that $\psi$ is fixed by the Rosati involution, we deduce that this decomposition is orthogonal with respect to the quadratic form $q$ on $\Lambda_\R$.

\smallskip
Recall that giving a Hodge structure of K3-type on $\Lambda$ is equivalent to giving an element $\omega$ in the period domain
\[
\Omega_\Lambda\coloneqq\{\omega\in \p(\Lambda_\C)\mid  q(\omega)=0, q(\omega,\overline{\omega})>0\}.
\]
Note that $\psi$ defines a morphism of Hodge structures if and only if $\omega$ is an eigenvector.
Therefore, as  $\Lambda_\R\simeq \Lambda_{\sqrt{d}}\oplus \Lambda_{-\sqrt{d}}$ is orthogonal with respect to $q$
and $(\Lambda^{2,0}\oplus\Lambda^{0,2})\cap \Lambda_\R$ has to be positive definite, there exists a Hodge structure for which $\psi$ is a Hodge morphism if and only if $\Lambda_{-\sqrt{d}}$ is negative definite and $\Lambda_{\sqrt{d}}$ has signature $(2,(\dim\Lambda)/2-2)$ or vice versa.
Let us assume that $\psi$ satisfy this hypothesis. Then, up to changing the sign of $\psi$, we may assume that $\Lambda_{\sqrt{d}}$ has signature $(2,(\dim\Lambda)/2-2)$.
We conclude that $\psi$ defines a Hodge automorphism if and only if the Hodge structure corresponds to an element in 
\[
\{\omega\in \p\left((\Lambda_{\sqrt{d}})_\C\right)\mid  q(\omega)=0, q(\omega,\overline{\omega})>0\}.
\]
Therefore, the locus of Hodge structures on $\Lambda$ for which $\psi$ defines a Hodge morphism has dimension equal to
$\dim \Lambda_{\sqrt{d}}-2=(\dim \Lambda)/2-2$.
\end{proof}
\end{prop}
\begin{rmk}
\label{rmk.: conditions on psi to be of Hodge type}
Let $\psi$ be a similarity of multiplier $d$ as in Proposition \ref{prop. dimension of locus admitting Hodge similarities}. From the proof of Proposition \ref{prop. dimension of locus admitting Hodge similarities}, we see that locus of Hodge structures of K3-type on $\Lambda$ for which $\psi$ is a Hodge similarity is non-empty (hence, of dimension $(\dim \Lambda)/2-2$) if and only if either $\Lambda_{\sqrt{d}}$ or $\Lambda_{-\sqrt{d}}$ is negative definite.
\end{rmk}
\smallskip
We use Proposition \ref{prop. dimension of locus admitting Hodge similarities} to show that the families of K3 surfaces which satisfy the hypotheses of Theorem \ref{thm.: algebraicity of sqrt p} are at most four-dimensional for $p=2$ and two-dimensional for $p=3$. Moreover, we produce examples of such families with these maximal dimensions. As we will see, no K3 surface satisfies the hypotheses of Theorem \ref{thm.: algebraicity of sqrt p} for higher values of $p$.

\smallskip
Let us start from the case $p=2$. Following \cite{van2007nikulin}, we call a symplectic involution on a K3 surface a \textit{Nikulin involution}.
By \cite[Prop.\ 2.2, 2.3]{van2007nikulin}, a K3 surface $X$ admits a Nikulin involution if and only if the lattice $E_8(-2)$ is primitively embedded in the N\'eron--Severi group of $X$. Note that, up to an automorphism of the K3-lattice, there exists a unique primitive embedding of $E_8(-2)$ in the K3-lattice. Therefore, we deduce from \cite[Sec.\ 1.3]{van2007nikulin} that $(E_8(-2))^{\perp}\simeq U^3\oplus E_8(-2)$. From this fact, we get the following criterion in terms of the transcendental lattice of $X$:
\begin{prop}
\label{prop.: criterion for Nikulin involution}
A K3 surface $X$ is Hodge isometric to a K3 surface admitting a Nikulin involution if and only if
$T(X)\subseteq U^3_\Q\oplus E_8(-2)_\Q$.\footnote{Thanks to G.\ Mezzedimi for the help with this argument.}
\begin{proof}
Let us first prove the ``only if" part. Let $X$ be a K3 surface such that $T(X)$ is Hodge isometric to $T(X')$ for some K3 surface $X'$ admitting a Nikulin involution.
By \cite[Prop.\ 2.2, 2.3]{van2007nikulin}, the lattice $E_8(-2)$ is primitively embedded in $\mathrm{NS}(X')$. Therefore, $\mathrm{NS}(X')^{\perp}\hookrightarrow E_8(-2)^{\perp}\simeq U^3\oplus E_8(-2)$.  Over $\Q$, we conclude that
\[
T(X)\simeq T(X')\hookrightarrow U^3_\Q\oplus E_8(-2)_\Q.
\]

\smallskip
For the ``if" part, let us assume that there is an embedding of quadratic spaces
\[
T(X)\hookrightarrow U^3_\Q\oplus E_8(-2)_\Q.
\]
Denote by $H^2(X,\Z)_{\mathrm{tr}}$ the transcendental part of the second integral cohomology of $X$. Clearing the denominators, we find a positive integer $\lambda\in \Z$ such that the above embedding restricts to an embedding of lattices
\[
j\colon \lambda H^2(X,\Z)_{\mathrm{tr}}\hookrightarrow U^3\oplus E_8(-2).
\]
Fix a primitive embedding $\iota\colon U^3\oplus E_8(-2)\hookrightarrow H^2(X,\Z)$ such that 
\[
\iota (U^3\oplus E_8(-2))^{\perp}\simeq E_8(-2).\]
Let $T'$ be the saturation of the lattice $(\iota \circ j)(H^2(X,\Z)_{\mathrm{tr}})\subseteq H^2(X,\Z)$.
For any K3 surface $X'$ such that $H^2(X',\Z)_{\mathrm{tr}}\simeq T'$ we get an embedding
\[
E_8(-2)\simeq (U^3\oplus E_8(-2))^{\perp} \hookrightarrow (T')^{\perp}\simeq \mathrm{NS}(X').
\]
This embedding is primitive, since $E_8(-2)$ is obtained as an orthogonal complement. Therefore, $X'$ admits a Nikulin involution by \cite[Prop.\ 2.2, 2.3]{van2007nikulin}.
Note that $T(X)$ and $T'_\Q$ are isometric quadratic spaces. Hence, by the surjectivity of the period map we can find a K3 surface $X'$ with $H^2(X',\Z)_{\mathrm{tr}}\simeq T'$ such that $T(X')$ is Hodge isometric to $T(X)$. This concludes the proof since the K3 surface $X'$ is Hodge isometric to $X$ and admits a Nikulin involution as required.
\end{proof}
\end{prop}
In particular, we deduce that the transcendental lattice of  a K3 surfaces which is Hodge isometric to a K3 surface with a Nikulin involution is at most $13$-dimensional. Proposition \ref{prop. dimension of locus admitting Hodge similarities} then shows that the families of K3 surfaces satisfying the hypotheses of Theorem \ref{thm.: algebraicity of sqrt p} in the case $p=2$ are at most four-dimensional. To prove the existence of such a four-dimensional family of K3 surfaces we consider a particular quadratic subspace of $U^3_\Q\oplus E_8(-2)_\Q$, and we show that it admits a similarity of multiplier $2$.
\begin{prop}
\label{prop. 4-dim family of K3 admitting a Nik inv}
The locus of Hodge structures of K3-type on $\Lambda\coloneqq U^2_\Q\oplus E_8(-2)_\Q$ which admit a Hodge similarity of multiplier $2$ which is fixed by the Rosati involution is non-empty and has a four-dimensional component.
\begin{proof}
As $\dim \Lambda=12$, Proposition \ref{prop. dimension of locus admitting Hodge similarities} shows that the locus of Hodge structures of K3-type on $\Lambda$ which admit a Hodge similarity of multiplier $2$ which is fixed by the Rosati involution has a four-dimensional component if non-empty. By Remark \ref{rmk.: conditions on psi to be of Hodge type}, we just need to produce a similarity $\psi$ of $\Lambda$ of multiplier $2$ fixed by the Rosati involution such that $\Lambda_{\sqrt{2}}$ has signature $(2,4)$.

\smallskip
As the quadratic space $E_8(-2)_\Q$ is isometric to $\langle -2\rangle^8$, we can write
$\Lambda=Q_1\oplus Q_2\oplus Q_3\oplus Q_4\oplus Q_5$,
with
\[
Q_1=Q_2\coloneqq \langle 1\rangle\oplus \langle -1\rangle,\quad
Q_3=Q_4=Q_5=Q_6\coloneqq \langle -2\rangle\oplus\langle -2\rangle.\]
As in \cite[Exmp.\ 3.4]{van2008real}, we restrict to finding a similarity $\psi$ which preserves the decomposition of $\Lambda$ as above. I.e., we look for matrices $M_i\in\mathrm{GL}_2(\Q)$ which satisfy the following: $^t{M}_i Q_i=Q_i M_i$ and $M_i^2=2 \mathrm{Id}$ for $i=1,\ldots, 6$. Then, $\psi\coloneqq M_1\oplus\ldots\oplus M_6$ will be fixed by the Rosati involution by the first condition and will be a similarity of multiplier $2$.
A direct computation shows that the following matrices satisfy all the above conditions
\[
M_1=M_2=
\begin{pmatrix}
\frac{3}{2} & -\frac{1}{2}\\
\frac{1}{2} & -\frac{3}{2}
\end{pmatrix}, \quad
M_3=M_4=M_5=M_6=
\begin{pmatrix}
1 & 1\\
1 & -1
\end{pmatrix},
\]
and that the signature of $\Lambda_{\sqrt{2}}$ is $(2,4)$. Thus, $\psi$ satisfies the required properties.
\end{proof}
\end{prop}

\begin{exmp}
\label{exmp: elliptic K3 with Nik inv}
By \cite[Sec.1\ 4]{van2007nikulin}, the family of elliptic K3 surfaces with a section and a two-torsion section provides an example of a ten-dimensional family of K3 surfaces with a Nikulin involution and general transcendental lattice $\Lambda=U^2_\Q\oplus \langle -2\rangle^8$. By Proposition \ref{prop. 4-dim family of K3 admitting a Nik inv}, there exists a four-dimensional family of elliptic K3 surfaces with a two-torsion section with endomorphism field containing $\Q(\sqrt{2})$.
\end{exmp}
\begin{rmk}
One can produce other examples of quadratic subspace of $U^3_\Q\oplus E_8(-2)^2_\Q$ which admit a similarity of multiplier $2$ which is fixed by the Rosati involution. For example, if $d>1$ is a square-free integer such that $2$ is a quadratic residue modulo $d$, the space $U^2_\Q\oplus \langle -2\rangle^7\oplus\langle -2d\rangle$ admits a similarity of multiplier $2$ and is not isometric to $U^2_\Q\oplus E_8(-2)_\Q$. This provides other four-dimensional families of K3 surfaces satisfying the hypotheses of Theorem \ref{thm.: algebraicity of sqrt p} for $p=2$.
\end{rmk}

To sum up, our discussion shows that Theorem \ref{thm.: algebraicity of sqrt p} in the case of Nikulin involutions gives the following:
\begin{thm}
\label{thm. Main thm for Nik. inv}
For every K3 surface in the four-dimensional families of K3 surfaces
with endomorphism field containing $\Q(\sqrt{2})$ which are Hodge isometric to a K3 surface with a Nikulin involution,
the endomorphism $\sqrt{2}$ is algebraic. 
In particular, the Hodge conjecture holds for the square of the general such K3 surface.
\end{thm}

\begin{rmk}
\label{rmk: obstructions on the endomorph. field}
The only case where Theorem \ref{thm. Main thm for Nik. inv} is not enough to prove the Hodge conjecture for the square of the K3 surfaces $X$ as in the statement is when the endomorphism field $E$ of $X$ is totally real of degree four and $T(X)$ is twelve-dimensional: 
this follows from the well known fact that, if the endomorphism field $E$ is totally real, the dimension of $T(X)$ as $E$-vector space is at least three. Recall that if $E$ is a CM field, then the Hodge conjecture for $X^2$ follows from \cite{buskin2019every} and \cite{huybrechts2019motives} using the fact that $E$ is generated by Hodge isometries.
Similarly to Proposition \ref{prop. dimension of locus admitting Hodge similarities}, one sees that the families of K3 surfaces as in Theorem \ref{thm. Main thm for Nik. inv} with totally real endomorphism field of degree four are one-dimensional.
\end{rmk}

Let us come to the case $p=3$. Let $X$ be a K3 surface with a symplectic automorphism of order $3$, and let $Y$ be the minimal resolution of the quotient. As above, we have an algebraic similarity $T(Y)\rightarrow T(X)$ of multiplier $3$ and $T(Y)$ is Hodge isometric to $T(X)(\frac{1}{3})$. By \cite[Thm.\ 4.1]{garbagnati2007symplectic}, a K3 surface $X$ admits a symplectic automorphism of order $3$ if and only if $K_{12}(-2)$ is primitively embedded in $\mathrm{NS}(X)$, where $K_{12}(-2)$ denotes the Coxeter--Todd lattice with the bilinear form multiplied by $-2$. With a similar proof as in Proposition \ref{prop.: criterion for Nikulin involution}, we can reformulate this in terms of the transcendental lattice as follows:
\begin{prop}
\label{prop.: criterion for symp. autom. of order 3}
A K3 surface $X$ is Hodge isometric to a K3 surface admitting a symplectic automorphism of order $3$ if and only if
$T(X)\subseteq U^3_\Q\oplus (A_2)_\Q^2$.\qed
\end{prop}

Proposition \ref{prop. dimension of locus admitting Hodge similarities} shows that families of K3 surfaces $X$ with $T(X)\subseteq U^3_\Q\oplus (A_2)^2_\Q$ whose endomorphism field contains $\Q(\sqrt{3})$ are at most two-dimensional.
As in the case of Nikulin involutions, we consider a particular quadratic subspace of $U^3_\Q\oplus (A_2)^2_\Q$, and we show that it admits a similarity of multiplier $3$:
\begin{prop}
\label{prop. existence of similarity of multiplier 3}
The locus of Hodge structures of K3-type on  $\Gamma\coloneqq U^2_\Q\oplus (A_2)^2_\Q$ which admit a Hodge similarity of multiplier $3$ which is fixed by the Rosati involution is non-empty and has a two-dimensional component.
\begin{proof}
As in the proof of Proposition \ref{prop. 4-dim family of K3 admitting a Nik inv}, we will construct an explicit similarity $\psi$ of $\Gamma$ of multiplier $3$ fixed by the Rosati involution such that $\Gamma_{\sqrt{3}}$ has signature $(2,2)$. Then, by Proposition \ref{prop. dimension of locus admitting Hodge similarities} and Remark \ref{rmk.: conditions on psi to be of Hodge type}, the locus of Hodge structures on $\Gamma$ for which $\psi$ is a Hodge morphism has a two-dimensional component.

\smallskip
Diagonalizing the quadratic space
$
(A_2)_\Q,
$
we see that there is an isometry
\[
\Gamma\simeq Q_1\oplus Q_2\oplus Q_3\oplus Q_4,
\]
with $Q_1=Q_2=\langle 1\rangle\oplus \langle -1\rangle$ and $Q_3=Q_4=\langle -2\rangle\oplus \langle -3/2\rangle$. We provide now matrices $M_1, M_2, M_3,M_4\in \mathrm{GL}_2(\Q)$ such that $^t M_iQ_i=Q_iM_i$ and $M_i^2=3 \mathrm{Id}$. As one checks, setting
\[
M_1=M_2\coloneqq
\begin{pmatrix}
    2 & -1\\
    1 & -2
\end{pmatrix}\text{ and } M_3=M_4\coloneqq
\begin{pmatrix}
0 & \frac{3}{2}\\
2 & 0
\end{pmatrix},
\]
the map 
$
\psi\coloneqq M_1\oplus M_2\oplus M_3\oplus M_4
$
defines a similarity of multiplier $3$ of $\Gamma$ satisfying all the requirements.
\end{proof}
\end{prop}

\begin{exmp}
By \cite[Prop.\ 4.2]{garbagnati2007symplectic}, the family of elliptic K3 surfaces with a section and a three-torsion section provides an example of a six-dimensional family of K3 surfaces with a symplectic automorphism of order $3$ and general transcendental lattice isometric to $\Gamma= U^2_\Q\oplus (A_2)^2_\Q$. By Proposition \ref{prop. existence of similarity of multiplier 3}, there is a two-dimensional subfamily of K3 surfaces with endomorphism field containing $\Q(\sqrt{3})$.
\end{exmp}
\begin{rmk}
As in the case of Nikulin involutions, one can construct other eight-dimensional quadratic subspaces of $U^3_\Q\oplus (A_2)^2_\Q$ admitting a similarity of multiplier $3$ as $\Gamma$ of Propostion \ref{prop. existence of similarity of multiplier 3}.
\end{rmk}
Our discussion shows that Theorem \ref{thm.: algebraicity of sqrt p} in the case $p=3$ gives the following:
\begin{thm}
\label{main thm for order three}
For every K3 surface in the two-dimensional families of K3 surfaces with endomorphism field containing $\Q(\sqrt{3})$ which are Hodge isometric to a K3 surface with a symplectic automorphism of order $3$, the endomorphism $\sqrt{3}$ is algebraic. In
particular, the Hodge conjecture holds for the square of every such K3 surface.
\end{thm}
\begin{rmk}
Note that in this case, Theorem \ref{main thm for order three} proves the Hodge conjecture for the square of every K3 surfaces of these families. The reason for this lies in the fact that the transcendental lattice of these K3 surfaces is at most eight-dimensional. Therefore, by a similar argument as in Remark \ref{rmk: obstructions on the endomorph. field}, we see that the endomorphism field of such K3 surface is either $\Q(\sqrt{3})$ or a CM field. In the latter case, the Hodge conjecture for the square of the K3 surface follows from the fact that CM fields are generated by Hodge isometries.
\end{rmk}

In the case of symplectic automorphisms of order bigger than $3$, the same procedure does not produce any K3 surface. In fact, the endomorphism of a K3 surface with a symplectic automorphism of order $5$ or $7$ is always $\Q$ or a CM field. This can be deduced from 
 \cite[Prop.\ 1.1]{garbagnati2007symplectic}: indeed, the transcendental lattice of a K3 surface admitting a symplectic automorphism of order $5$ is of dimension at most five, and for K3 surfaces with a symplectic automorphism of order $7$ its dimension is at most three. As in Remark \ref{rmk: obstructions on the endomorph. field}, one sees that, in both cases, the endomorphism field of these K3 surfaces cannot be a totally real field different from $\Q$.

\section{Kuga--Satake varieties and Hodge similarities}
\label{sec. KS varieties}
By a construction due to Kuga and Satake \cite{KS67}, given a polarized Hodge structure of K3-type $(V,q)$, there exists an abelian variety $\KS(V)$, called the Kuga--Satake variety of $(V,q)$, together with an embedding of Hodge structures
$\kappa\colon V\hookrightarrow H^1(\KS(V),\Q)^{\otimes 2}.$ We refer the reader for this construction to \cite[Ch.\ 4]{huybrechts2016lectures}, \cite{VG00}, and \cite{varesco2022hyper}.
In this section, we prove the functoriality of the Kuga--Satake construction with respect to Hodge similarities:
\begin{prop}
\label{prop. similarities and KS}
Let $\psi\colon (V,q)\rightarrow (V',q')$ be a Hodge similarity of polarized Hodge structures of K3-type. Then, there exists an isogeny of abelian varieties $\psi_{\KS}\colon \mathrm{KS}(V)\rightarrow\KS(V')$ making the following diagram commute
\[
\begin{tikzcd}
V \arrow[d, hook] \arrow[r, "\psi"] & V' \arrow[d, hook] \\
H^1(\KS(V),\Q)^{\otimes 2} \arrow[r, "({\psi_{\KS}})_*^{\otimes 2}"]      &  H^1(\KS(V'),\Q)^{\otimes 2}
\end{tikzcd},
\]
where the vertical arrows are the Kuga--Satake correspondence for $(V,q)$ and $(V',q')$.
\end{prop}
In the remainder of this section we prove Proposition \ref{prop. similarities and KS}.

\smallskip
Let $(V,q)$ and $(V',q')$ be polarized Hodge structures of K3-type, and let $\psi\colon (V,q)\rightarrow (V',q')$ be a Hodge similarity of multiplier $\la$. The next lemma shows that $\psi$ induces an isomorphism between the even Clifford algebras $\Cl^+(V)$ and $\Cl^+(V')$. Recall that $\Cl^+(V)$ is defined as the even degree part of $\bigotimes^*V/I_V$, where $I_V$ is the two-sided ideal generated by elements of the form $v\otimes v-q(v)$, for $v\in V$.
\begin{lem}
\label{lem. induced morph. on even Clifford algebras}
The isomorphism of graded rings 
\[
\textstyle
\psi_{\otimes}\colon \bigotimes^{\mathrm{ev}} V\rightarrow \bigotimes^{\mathrm{ev}} V', \quad v_1\otimes\cdots\otimes v_{2m}\mapsto (1/\la)^{m}\psi v_1\otimes\cdots\otimes \psi v_{2m}
\] 
induces an isomorphism $\psi_{\Cl}\colon \Cl^+(V)\xrightarrow{\simeq}\Cl^+(V')$.
\begin{proof}
From the definition, it is immediate to see that the map $\psi_\otimes$ is an isomorphism of graded rings. 
Given $v\in V$, we have the following 
\begin{align*}
    \psi_\otimes(v\otimes v-q(v))&=(1/\la) (\psi v\otimes \psi v)-q(v)= (1/\la) (\psi v\otimes \psi v-\la q(v))\\
    &=(1/\la)(\psi v\otimes \psi v-q'(\psi v)),
\end{align*}
where in the last step we used that $\psi$ is a similarity of multiplier $\la$.
This equality shows that $\psi_\otimes(v\otimes v-q(v))$ belongs to the ideal of $\bigotimes^{\mathrm{ev}} V'$ generated by $w\otimes w-q'(w)$ for $w\in V'$. Hence, the isomorphism $\psi_\otimes$ descends to an isomorphism $\psi_{\Cl}\colon \Cl^+(V)\rightarrow\Cl^+(V')$.
\end{proof}
\end{lem}
In the construction of the Kuga--Satake variety associated to a polarized Hodge structure of K3-type, the complex Hodge structure on $\Cl^+(V)_\R$ is given by left multiplication by $J\coloneqq e_1\cdot e_2$, with $\{e_1,e_2\}$ an orthogonal basis of $V_\R\cap( V^{2,0}\oplus V^{0,2})$ satisfying $q(e_1)=q(e_2)=-1$. As one checks, this complex structure does not depend on the choice of the basis.
\begin{lem}
\label{lem. psi is a morph of complex structures}
The map $\psi_{\Cl,\R}\colon \Cl^+(V)_\R\rightarrow\Cl^+(V')_\R$ is compatible with the natural complex structures on $\Cl^+(V)_\R$ and $\Cl^+(V')_\R$.
\begin{proof}
Let $\{e_1,e_2\}$ be an orthogonal basis of $V_\R\cap (V^{2,0}\oplus V^{0,2})$ with $q(e_i)=-1$, and define 
\[e_i'\coloneqq \psi_\R e_i/\sqrt{\la}\in V'_\R\cap (V'^{2,0}\oplus V'^{0,2}), \quad \text{for } i=1,2.\]
As $\psi$ is a Hodge similarity of multiplier $\lambda_\psi$, one sees that $\{e'_1,e'_2\}$ is an orthogonal basis of $V'_\R\cap (V'^{2,0}\oplus V'^{0,2})$ such that $q'(e_i')=-1$.
The complex structure on $\Cl^+(V)_\R$ (resp., on $\Cl^+(V')_\R$) is then induced by left multiplication by $J\coloneqq e_1\cdot e_2$ (resp., by $J'\coloneqq e'_1\cdot e'_2$). Hence, the equality 
\[
\psi_{\Cl,\R}(J\cdot x)=J'\cdot \psi_{\Cl,\R} (x) \quad \forall x\in \Cl^+(V)
\]
proves that $\psi_{\Cl,\R}$ is a morphism of complex vector spaces.
\end{proof}
\end{lem}
The Kuga--Satake variety of $(V,q)$ is defined as (the isogeny class) of the complex torus $\KS(V)\coloneqq \Cl^+(V)_\R/\Cl^+(V)$, where $\Cl^+(V)_\R$ is endowed with the complex structure we recalled above. Lemma \ref{lem. psi is a morph of complex structures} then shows that $\psi_{\Cl}\colon \Cl^+(V)\rightarrow\Cl^+(V')$ induces an isogeny of complex tori
\[
\psi_\KS\colon \KS(V)\rightarrow \KS(V').
\]
Recall that Kuga--Satake varieties of polarized Hodge structures of K3-type are abelian varieties: let $(f_1,f_2)\in V\times V$ be a pair of orthogonal elements of $V$ with positive square, and consider the pairing
\[
Q\colon \Cl^+(V)\times \Cl^+(V)\rightarrow \Q, \quad (v,w)\rightarrow \mathrm{tr}(f_1\cdot f_2\cdot v^*\cdot w),
\]
where $\mathrm{tr}(x)$ denotes the trace of the endomorphism of $\Cl^+(V)$ given by left multiplication by $x\in \Cl^+(V)$ and $v^*$ denotes the image of $v$ under the involution of $\Cl^+(V)$ induced by the involution $v_1\otimes\cdots \otimes v_{2m}\mapsto v_{2m}\otimes\cdots \otimes v_1$ on $\bigotimes^{\mathrm{ev}}V$. Then, $Q$ defines up to a sign a polarization for the weight-one Hodge structure $\Cl^+(V)$.
Note that the pair $(\psi f_1/\lambda_\psi ,\psi f_2 )\in V'\times V'$ satisfies the same hypotheses. Hence, it defines a polarization $Q'$ on $\Cl^+(V')$.
\begin{lem}
\label{lem.: existence of isogeny of KS}
The isomorphism $\psi_\Cl\colon\Cl^+(V)\rightarrow \Cl^+(V')$ is compatible with the polarizations $Q$ and $Q'$ defined above. Hence, $\psi_{\KS}\colon\KS(V)\rightarrow\KS(V')$ is an isogeny of abelian varieties.
\begin{proof}
To prove the lemma, we need to show that
\[
Q(v,w)=Q'(\psi_\Cl v,\psi_\Cl w),
\]
for all $v,w \in \Cl^+(V)$. By definition of $Q$ and $Q'$, this is equivalent to prove that
\[
\mathrm{tr}(f_1\cdot f_2\cdot v^* \cdot w) = \mathrm{tr}(\psi f_1/\lambda_\psi\cdot \psi f_2\cdot (\psi_\Cl v)^* \cdot \psi_\Cl w).
\]
Note that, by definition of $\psi_\Cl$, the following holds
\[
\psi f_1/\lambda_\psi\cdot \psi f_2\cdot (\psi_\Cl v)^* \cdot \psi_\Cl w=\psi_\Cl(f_1\cdot f_2\cdot v^* \cdot w).
\]
We then need to prove that, for every $x\in \Cl^+(V)$, the left multiplication by $x$ on $\Cl^+(V)$ has the same trace as the left multiplication by $\psi_\Cl x$ on $\Cl^+(V')$. This can be checked as follows: let $\{b_i\}_i$ be a basis of $\Cl^+(V)$ with dual basis $\{b^i\}_i$. By definition, we have that
\[
\textstyle\mathrm{tr}(x)=\Sigma_i b^i(x\cdot b_i)
\]
As a basis of $\Cl^+(V')$, consider $\{\psi_\Cl b_i\}_i$. Its dual basis is $\{ \psi_\Cl^{\vee}b^i\}_i$, where $\psi_\Cl^{\vee}$ is the dual action of $\psi_\Cl$. The trace of the left multiplication by $\psi_\Cl x$ is then
\[
\mathrm{tr}(\psi_\Cl x)=\textstyle \Sigma_i (\psi_\Cl^{\vee} b^i)(\psi_\Cl x\cdot \psi_\Cl b_i)=\Sigma_i b^i(x\cdot b_i)=\mathrm{tr}(x).
\]
This concludes the proof.
\end{proof}
\end{lem}
The last ingredient for the proof of Proposition \ref{prop. similarities and KS} is the compatibility of the isomorphism $\psi_\Cl$ with the embedding $\varphi_V\colon  V\hookrightarrow \mathrm{End}(\Cl^+(V))$ given by the Kuga--Satake construction. Recall that $\varphi_V$ is given as follows: let $v_0\in V$ be an element with $q(v_0)\not= 0$, then $\varphi_V(v)\coloneqq f_v\in \mathrm{End}(\Cl^+(V))$
where $f_v(w)\coloneqq v\cdot w\cdot v_0$. Similarly, let $\varphi_{V'}\colon  V'\hookrightarrow \mathrm{End}(\Cl^+(V'))$ be the embedding corresponding to the element $\psi(v_0)/\lambda_\psi\in V'$.
\begin{lem}
With the previous notation, the following diagram commutes 
\[
\begin{tikzcd}
V \arrow[d, hook, "\varphi_V"] \arrow[r, "\psi"] & V' \arrow[d, hook, "\varphi_{V'}"] \\
\mathrm{End}(\Cl^+(V)) \arrow[r, "\mathrm{End}(\psi_{\Cl})"]      &  \mathrm{End}(\Cl^+(V'))
\end{tikzcd},
\]
where $\mathrm{End}(\psi_{\Cl})$ is the map $f\mapsto \psi_\Cl\circ f\circ \psi_\Cl^{-1}$.
\begin{proof}
By definition, the composition of $\varphi_V$ with $\mathrm{End}(\psi_{\Cl})$ is the map 
\[
V\rightarrow \mathrm{End}(\mathrm{Cl}^+(V')), \quad v\mapsto (w'\mapsto \psi(v)\cdot w'\cdot \psi(v_0)/\lambda_{\psi}).
\]
This shows that the above square is commutative. Indeed, $\varphi_{V'}$ is the map
\[
V'\hookrightarrow \mathrm{End}(\mathrm{Cl}^+(V')),\quad v'\mapsto (w'\mapsto v'\cdot w'\cdot \psi (v_0)/\lambda_\psi).\qedhere
\]
\end{proof}
\end{lem}
\begin{proof}[Proof of Propostion \ref{prop. similarities and KS}]
Lemma \ref{lem.: existence of isogeny of KS} shows that there exists an isogeny of abelian varieties $\psi_\KS\colon \KS(V)\rightarrow \KS(V')$ such that $(\psi_\KS)_*=\psi_\Cl\colon \mathrm{Cl}^+(V)\rightarrow\mathrm{Cl}^+(V')$. 
Recall that the Kuga--Satake embedding is the composition
\[
V\xhookrightarrow{\varphi_V}\mathrm{End}(\Cl^+(V))\simeq \Cl^+(V)\otimes \Cl^+(V),
\]
where the isomorphism is given by the polarization $Q$ on $\Cl^+(V)$ which induces an isomorphism between $\Cl^+(V)$ and $\Cl^+(V)^*$.
Note that the commutativity of the square in the theorem follows from the commutativity of the square of Lemma \ref{lem.: commutative square} by the compatibility of $\psi_\Cl$ with the polarizations $Q$ and $Q'$.
\end{proof}
\subsection{De Rham--Betti similarities}
\label{subsection dRB similarities}
In \cite[Thm.\ 9.5]{vial2023dRBconj}, the authors prove that de Rham--Betti isometries between the second de Rham--Betti cohomology groups of two hyperk\"ahler manifolds defined over $\overline{\Q}$ are motivated in the sense of Andr\'e using the fact that the Kuga--Satake correspondence is motivated.
We note here that the observation that similarities between two quadratic spaces induce isomorphisms between the respective even Clifford algebras shows that the result in \cite{vial2023dRBconj} can be extended to de Rham--Betti similarities.
\smallskip

Let us briefly recall the notions of de Rham--Betti morphism and of motivated cycles as presented in \cite{vial2023dRBconj}. To simplify the exposition, we avoid going into too much detail of the Tannakian formalism.

\smallskip
\begin{defn}
Let $X$ be a smooth projective variety over $\overline{\Q}$, and let $\Q(k)\coloneqq (2\pi i)^k\Q\subseteq \C$. The \textit{de Rham--Betti cohomology groups} of $X$ are the triples
\[
H^n_{\dRB}(X,\Q(k))\coloneqq (H^n_{\mathrm{dR}}(X/\overline{\Q}), H^n_{\mathrm{B}}(X_\C,\Q(k)),c_X),
\]
where 
\[
c_X\colon H^n_{\mathrm{dR}}(X/\overline{\Q})\otimes_{\overline{\Q}} \C\xrightarrow{\simeq} H^n_{\mathrm{B}}(X_\C,\Q(k))\otimes_\Q\C
\] is the Grothendieck's period comparison isomorphism.
Given $X'$ another smooth projective variety over $\overline{\Q}$, a \textit{de Rham--Betti morphism} between $H_{\mathrm{dRB}}^n(X,\Q(k))$ and $H_{\mathrm{dRB}}^n(X',\Q(k))$ consists of a pair of morphisms
\[
f_{\mathrm{dR}}\colon H^n_{\mathrm{dR}},(X/\overline{\Q})\rightarrow H^n_{\mathrm{dR}}(X'/\overline{\Q}) \quad \text{and} \quad f_{B}\colon H^n_{\mathrm{B}}(X_\C,\Q(k))\rightarrow H^n_{\mathrm{B}}(X'_\C,\Q(k)),
\]
where $f_{\mathrm{dR}}$ is $\overline{\Q}$-linear and $f_{B}$ is $\Q$-linear, and their $\C$-linear extensions are compatible with $c_X$ and $c_{X'}$.
\end{defn}
\begin{defn}
Let $X$ be a smooth projective variety over $\overline{\Q}$. 
A \textit{motivated cycle} on $X$ is an element of $H^{2r}_{\mathrm{B}}(X_\C,\Q(r))$ of the form $p_{X,*}(\alpha\cup *_L \beta)$, where $\alpha$ and $\beta$ are algebraic cycles on $X\times_{\overline{\Q}} Y$ for some smooth projective variety $Y$ over $\overline{\Q}$, $*_L$ is the (inverse of the) Lefschetz isomorphism, and $p_{X,*}$ is the first projection.
Note that, given a motivated cycle $\alpha_{\mathrm{B}}$ in $H_{\mathrm{B}}^{2k}(X, \Q(k))$, there exists a de Rham cohomology class $\alpha_{\mathrm{dR}}$ in $H^{2k}_{\mathrm{dR}}(X/\overline{\Q})$ such that $c_X(\alpha_{\mathrm{dR}})=\alpha_{\mathrm{B}}$. In particular, we see that a motivated cycle on $X\times_{\overline{\Q}} Y$ induces a de Rham--Betti morphism between the de Rham--Betti cohomologies of $X$ and $Y$. 
One says that a de Rham--Betti morphism between cohomology groups of two smooth projective varieties $X$ and $Y$ over $\overline{\Q}$ is \textit{motivated} if it is induced by a motivated cycle on $X\times_{\overline{\Q}} Y$. For a complete introduction on this subject we refer the reader to \cite{andre2004motifs}. 
\end{defn}
Following \cite{vial2023dRBconj}, a variety over $\overline{\Q}$ is called a \textit{hyperk\"ahler manifold over} $\overline{\Q}$ if its base-change to $\C$ is a hyperk\"ahler manifold with second Betti number at least three. This last assumption is needed to ensure that the Kuga--Satake correspondence is motivated as proved by André \cite{andre1996shafarevich}. The following proposition extends the result of \cite[Thm.\ 9.5]{vial2023dRBconj}.
\begin{prop}
\label{prop.: dRB similarities}
Let $X$ and $X'$ be hyperk\"ahler manifolds over $\overline{\Q}$. Then, any de Rham--Betti similarity $H_{\mathrm{dRB}}^2(X,\Q)\xrightarrow{\simeq}H_{\mathrm{dRB}}^2(X',\Q)$ is motivated.
\begin{proof}
The proof of this theorem is exactly the same as the one in the reference with the only addition that similarities (and not just isometries) induce isomorphisms between the even Clifford algebras. We give here just a sketch of the proof.

\smallskip
Let $H_{\mathrm{dRB}}^2(X,\overline{\Q})$ be the second $\overline{\Q}$-de Rham--Betti cohomology group of $X$. That is
\[
H_{\mathrm{dRB}}^2(X,\overline{\Q})\coloneqq (H^2_{\mathrm{dR}}(X/\overline{\Q}),H^2_{\mathrm{B}}(X_\C,\overline{\Q}),c_X).
\]
Similarly define $H_{\mathrm{dRB}}^2(X',\overline{\Q})$.
By \cite[Lem.\ 6.2, 6.17]{vial2023dRBconj}, to prove the theorem, it suffices to prove the same statement over $\overline{\Q}$. I.e., that every $\overline{\Q}$-de Rham--Betti similarity between $
H_{\mathrm{dRB}}^2(X,\overline{\Q})$ and $ H_{\mathrm{dRB}}^2(X',\overline{\Q})$
is a $\overline{\Q}$-linear combination of motivated cycles on $X\times_{\overline{\Q}}X'$.

\smallskip
Let $T_{\dRB}^2(X,\overline{\Q})$ be the orthogonal complement of the subspace of $H_{\dRB}^2(X,\overline{\Q})$ spanned by divisor classes, and similarly define $T_{\dRB}^2(X',\overline{\Q})$.
As in the reference, one shows that, to prove the result, it suffices to show that every de Rham--Betti similarity between $ T_{\mathrm{dRB}}^2(X,\overline{\Q})$ and $T_{\mathrm{dRB}}^2(X',\overline{\Q})$ is $\overline{\Q}$-motivated.

\smallskip
Consider the $\overline{\Q}$-linear category $\mathscr{C}_{\overline{\Q}-\mathrm{dRB}}$ whose objects are triples $(M_{\mathrm{dR}},M_{\mathrm{B}},c_M),$ where $M_{\mathrm{dR}}$ and $M_{\mathrm{B}}$ are finite dimensional $\overline{\Q}$-vector spaces and \[c_M\colon M_{\mathrm{dR}}\otimes_{\overline{\Q}} \C\rightarrow M_{\mathrm{B}}\otimes_{\overline{\Q}} \C\] is a $\C$-linear isomorphism.  As in  \cite[Sec.\ 4.3]{vial2023dRBconj}, one sees that $\mathscr{C}_{\overline{\Q}-\mathrm{dRB}}$ is a neutral Tannakian category. Denote by $G_{\overline{\Q}-\mathrm{dRB}}$ its Tannakian fundamental group.
Note that $T_{\mathrm{dRB}}^2(X,\overline{\Q})$ and $T_{\mathrm{dRB}}^2(X',\overline{\Q})$ are objects in $\mathscr{C}_{\overline{\Q}-\mathrm{dRB}}$.

\smallskip
Denote by $V\coloneqq T_{\mathrm{B}}^2(X,\Q(1))$ and $V'\coloneqq T_{\mathrm{B}}^2(X',\Q(1))$ the transcendental Betti cohomologies of $X$ and $X'$.
Given a $\overline{\Q}$-de Rham--Betti similarity $\psi_{\dRB}\colon T_{\mathrm{dRB}}^2(X,\overline{\Q})\xrightarrow{\simeq}T_{\mathrm{dRB}}^2(X',\overline{\Q})$, it induces a similarity
\[
\psi\colon V \otimes \overline{\Q}\rightarrow  V' \otimes \overline{\Q}.
\]
As $\psi_\dRB$ is a morphism in $\mathscr{C}_{\overline{\Q}-\mathrm{dRB}}$, the morphism $\psi$ is $G_{\overline{\Q}-\mathrm{dRB}}$-invariant by the Tannakian formalism.
With the same definition as in Lemma \ref{lem. induced morph. on even Clifford algebras}, we see that $\psi$ induces a $G_{\overline{\Q}-\mathrm{dRB}}$-invariant isomorphism of algebras 
\[
\psi_{\Cl}\colon \mathrm{Cl}^+(V\otimes\overline{\Q})\rightarrow \mathrm{Cl}^+(V'\otimes\overline{\Q}).
\]
As in the reference, one then shows that this induces a  $G_{\overline{\Q}-\mathrm{dRB}}$-invariant isomorphism of algebras $J\colon \mathrm{End}(\Cl(V)\otimes\overline{\Q})\rightarrow \mathrm{End}(\Cl(V')\otimes\overline{\Q})$.
One then shows that $J$ is $\overline{\Q}$-motivated.
This in turn implies that $\psi$ is $\overline{\Q}$-motivated using the fact that the Kuga--Satake correspondence is $\overline{\Q}$-motivated as proven in \cite[Prop.\ 8.5]{vial2023dRBconj}. This concludes the proof.
\end{proof}
\end{prop}

\section{Hodge similarities and algebraic classes}
\label{Sec. implication of functoriality of KS}
We now go back to the case of hyperk\"ahler manifolds defined over $\C$ and study the consequences of the functoriality of the Kuga--Satake construction relative to Hodge similarities in the case where the Hodge structure $(V,q)$ is geometrical. In other words, we assume that there is a hyperk\"ahler manifold $X$ for which $V=T(X)$ or $V=H^2(X,\Q)$ and $q$ is the Beauville--Bogomolov quadratic form with the sign changed.
\begin{rmk}
In Section \ref{sec. KS varieties}, we studied the Kuga--Satake construction for polarized Hodge structures of K3-type. The same construction also works for the second cohomology group of a hyperk\"ahler manifold $X$ even though it is not polarized by the Beauville--Bogomolov quadratic form. Indeed, using the direct sum decomposition $H^2(X,\Q)\simeq T(X)\oplus \mathrm{NS}(X)_\Q$, one sees that the even Clifford algebra of $H^2(X,\Q)$ is a power of the even Clifford algebra of $T(X)$. Thus, the Kuga--Satake variety $\mathrm{KS}(H^2(X,\Q))$ is an abelian variety isogenous to a power of $\mathrm{KS}(T(X))$.
\end{rmk}
\smallskip
Let $\KS(X)$ be the Kuga--Satake variety of $H^2(X,\Q)$. The Kuga--Satake correspondence gives an embedding of Hodge structures:
\[
\kappa_X\colon H^2(X,\Q)\hookrightarrow H^1(\KS(X),\Q)^{\otimes 2}\subseteq H^2(\KS(X)^2,\Q).
\]
The Hodge conjecture predicts that $\kappa_X$ is algebraic:
\begin{conj}
[Kuga--Satake Hodge conjecture]
\label{KS hodge conj}
Let $X$ be a hyperk\"ahler manifold, then, the Kuga--Satake correspondence $\kappa_X$ is algebraic.
\end{conj}
\begin{rmk}
\label{rmk.: dependence of KS corr. on the coices of three vectors}
Note that the Kuga--Satake correspondence depends on the choice of the three elements $v_0,f_1,f_2\in T(X)$ as in Section \ref{sec. KS varieties}. Choosing a different $\widetilde{v}_0\in T(X)$ changes the embedding by the automorphism of $\Cl^+(H^2(X,\Q))$ which sends $w$ to $\frac{ w\cdot v_0\cdot \widetilde{v}_0}{q(v_0)}$, and choosing a different pair $\widetilde{f}_1,\widetilde{f}_2\in T(X)$ corresponds to changing the polarization on the complex torus $\KS(X)$. However, neither of these two operations affects the algebraicity of $\kappa_X$. Hence, Conjecture \ref{KS hodge conj} does not depend on the choices made in the definition of $\kappa_X$.
\end{rmk}
Let $2n\coloneqq \dim X$ and $N\coloneqq\dim \KS(X)$. The transpose of $\kappa_X^{\vee}$ of $\kappa_X$ is the surjection 
\[
\kappa_X^{\vee}\colon  H^{4N-2}(\KS(X)^2,\Q)\twoheadrightarrow H^{4n-2}(X,\Q).
\]
Note that, as $\kappa_X$ and $\kappa_X^{\vee}$ are transpose of each other, $\kappa_X$ is algebraic if and only if $\kappa_X^{\vee}$ is algebraic. 
Let $h_X\in H^2(X,\Q)$ be the cohomology class of an ample divisor on $X$. By the strong Lefschetz theorem, the cup product with $h_X^{2n-1}$ induces an isomorphism of Hodge structures
\[
h_X^{2n-1}\cup\bullet\colon H^2(X,\Q)\rightarrow H^{4n-2}(X,\Q)(2n-2),
\]
where $(2n-2)$ denotes the Tate twist by $\Q(2n-2)$.
Let 
\[
\delta_X\coloneqq (h_X^{2n-1}\cup\bullet )^{-1}\colon H^{4n-2}(X,\Q)(2n-2)\rightarrow H^2(X,\Q)
\]
be the inverse map. As $h_X^{2n-1}\cup\bullet$ is an isomorphism of Hodge structures, also $\delta_X$ is an isomorphism of Hodge structures. Note that it is in general not known whether $\delta_X$ is algebraic or not. 
Adapting the proof of \cite[Lem.\ 3.4]{voisin2014bloch}, we show that $\kappa_X$ and $\kappa_X^{\vee}$ satisfy the following:
\begin{lem}
\label{lem.: commutative square}
Let $X$ be a hyperk\"ahler manifold of dimension $2n$, and let $h_{\KS}\in H^2(\KS(X),\Q)$ be the class of an ample divisor on $\KS(X)$. Denote by $\varphi$ the restriction to $T(X)$ of the composition 
\[
\delta_X\circ {\kappa_X^{\vee}}\circ (h_{\KS}^{2N-2}\cup\bullet )\circ \kappa_X\colon H^2(X,\Q)\rightarrow H^2(X,\Q).
\]
Then, $\varphi$ is a non-zero rational multiple of the identity $\mathrm{Id}_{T(X)}\colon T(X)\rightarrow T(X)$.
\begin{proof}
Let us begin by showing that $\varphi$ is a Hodge automorphism. Note that $\varphi$ is by construction a morphism of Hodge structures. It is then an element of the endomorphism field of $T(X)$. As $T(X)$ is an irreducible Hodge structure, every non-trivial endomorphism is an automorphism. In particular, we just need to show that $\varphi$ is non-zero. As $\delta_X\colon H^{4n-2}(X,\Q)(2n-2)\rightarrow H^2(X,\Q)$ is an isomorphism of Hodge structures, it suffices to show that the map 
\[
\left. \left(\kappa_X^{\vee} \circ (h_{\KS}^{2N-2}\cup\bullet)\circ \kappa_X\right)\right\rvert_{T(X)}\colon T(X)\rightarrow H^{4n-2}(X,\Q)
\]
is non-zero. Let $\omega\in H^{2,0}(X)$ be the class of a symplectic form. As $\kappa_X$ and $\kappa_X^{\vee}$ are adjoint with respect to the Hodge--Riemann pairing, we have the following equality
\[
\langle \omega,{\kappa_X^{\vee}} \circ (h_{\KS}^{2N-2}\cup\bullet)\circ \kappa_X(\overline{\omega})\rangle_X=\langle \kappa_X (\omega),(h_{\KS}^{2N-2}\cup\bullet)\circ \kappa_X(\overline{\omega})\rangle_{\KS(X)}.
\]
The right-hand side is non-zero by the Hodge--Riemann relations, as $0\not =\kappa_X(\omega)\in H^{2,0}(\KS(X))$ by the injectivity of $\kappa_X$. In particular, we conclude that $\kappa_X^{\vee} \circ (h_{\KS}^{2N-2}\cup\bullet)\circ \kappa_X(\overline{\omega})\not =0$. This implies that $\kappa_X^{\vee} \circ (h_{\KS}^{2N-2} \cup \bullet)\circ \kappa_X$ restricted to $T(X)$ is non-zero and proves the first statement of the lemma.

\smallskip
To prove that $\varphi$ is a rational multiple of the identity, let us first assume that $X$ is Mumford--Tate general. In this case, 
$\mathrm{End}_{\mathrm{Hdg}}(T(X))=\Q$. Hence, the statement is obvious since every Hodge automorphism is a rational multiple of the identity. 
For the special case, just note that it is possible to deform in the moduli space of polarized hyperk\"ahler manifolds the pair $(X,h_X)$ to a pair $(X',h_{X'})$, where $X'$ is Mumford--Tate general. Then, as all the maps involved in the definition of $\varphi$ deform in families, the statement for $(X', h_{X'})$ readily implies the statement for $(X, h_X)$.
\end{proof}
\end{lem}
We finally have all tools to show the main theorem of this section:
\begin{thm}
\label{main thm}
Let $X'$ and $X$ be two hyperk\"ahler manifolds for which the Kuga--Satake Hodge conjecture holds.
Then, for every Hodge similarity $\psi\colon T(X')\rightarrow T(X)$, the composition 
\[
T(X')\xlongrightarrow{\psi} T(X)\xrightarrow{h^{2n-2}_X\cup\bullet} H^{4n-2}(X,\Q)
\]
is algebraic, where $2n\coloneqq \dim X$.
\begin{proof}
By the functoriality of the Kuga--Satake correspondence of Propositon \ref{prop. similarities and KS}, the similarity $\psi\colon T(X')\rightarrow T(X)$ induces an isogeny $\psi_{\KS}\colon\KS(X')\rightarrow\KS(X)$ such that the induced isomorphism 
\[
(\psi_{\KS})^{\otimes 2}_*\colon H^1(\KS(X'),\Q)^{\otimes2}\rightarrow H^1(\KS(X),\Q)^{\otimes2}
\]
makes the following diagram commute:
\begin{equation}
\label{Square 1}
\begin{tikzcd}
T(X') \arrow[d, hook, "\kappa_{X'}"] \arrow[r, "\psi"]  & T(X) \arrow[d, hook, "\kappa_X"] \\
{H^1(\KS(X'),\Q)^{\otimes2}} \arrow[r, "(\psi_{\KS})^{\otimes 2}_*"] & {H^1(\KS(X),\Q)^{\otimes2}}   
\end{tikzcd},
\end{equation}
where $\kappa_{X'}$ (resp., $\kappa_{X}$) is the Kuga--Satake correspondence for $T(X')$ (resp., $T(X))$.
By Lemma \ref{lem.: commutative square}, the automorphism $\varphi$ of $T(X)$ makes the following diagram commute
\begin{equation}
\label{Square 2}
\begin{tikzcd}
T(X) \arrow[d, "\kappa_{X} ", hook] \arrow[r, "(h_X^{2n-2}\cup \bullet)\circ\varphi"] & H^{4n-2}(X,\Q) \\
{H^2(\KS(X)^2,\Q)} \arrow[r, "h_{\KS}^{2N-2}\cup \bullet"]  & {H^{2N-2}(\KS(X)^2,\Q)} \arrow[u, "{\kappa_X^{\vee}} ", two heads]
\end{tikzcd}.
\end{equation}
The commutativity of the squares (\ref{Square 1}) and (\ref{Square 2}) implies that the following equality holds
\[
(h_X^{2n-2}\cup \bullet)\circ\varphi\circ \psi= {\kappa_X^{\vee}}\circ (h_{\KS}^{N-2}\cup \bullet)\circ (\psi_{\KS})^{\otimes 2}_*\circ \kappa_{X'}.
\]
Note that the right-hand side of the equality is algebraic: $(\psi_{\KS})^{\otimes 2}_*$ is induced by an isogeny of abelian varieties, and $\kappa_X^{\vee}$ and $\kappa_{X'}$ are algebraic by assumption. We then conclude that the composition \[(h_X^{2n-2}\cup \bullet)\circ\varphi\circ \psi\colon T(X')\rightarrow H^{4n-2}(X,\Q)\] is algebraic as well.
I.e., there exists a cycle $\Gamma\in \mathrm{CH}^*(X'\times X)$ inducing the morphism
\[
[\Gamma]_*=(h_X^{2n-2}\cup \bullet)\circ \varphi\circ \psi\colon T(X')\rightarrow H^{4n-2}(X,\Q).
\]
By Lemma \ref{lem.: commutative square}, the automorphism $\varphi$ is by equal to $\lambda \mathrm{Id}_{T(X)}$ for some non-zero $\lambda\in \Q$. Therefore, the class $\Gamma/\lambda\in \mathrm{CH}^*(X'\times X)$ induces the morphism 
\[
T(X')\xrightarrow{\psi} T(X)\xrightarrow{h_X^{2n-2}\cup \bullet} H^{4n-2}(X,\Q).
\]
This concludes the proof.
\end{proof}
\end{thm}
This shows that, if the Kuga--Satake correspondence is algebraic, then every Hodge similarity is algebraic after composing it with the Lefschetz isomorphism. The Lefschetz standard conjecture in degree two for $X$ predicts that the inverse of
$h_X^{2n-2}\cup\bullet\colon H^2(X,\Q)\rightarrow H^{4n-2}(X,\Q)$
is algebraic.
If $X$ satisfies this conjecture, Theorem \ref{main thm} gives the following:
\begin{cor}
\label{main cor}
Let $X$ and $X'$ be hyperk\"ahler manifolds satisfying the Kuga--Satake Hodge conjecture. Assume moreover that $X$ satisfies the  Lefschetz standard conjecture in degree two. Then, every Hodge similarity $\psi\colon T(X')\rightarrow T(X)$ is algebraic. \qed
\end{cor}
\section{Applications}
\label{Sec. applications}
In this section, we recall the main cases where the Kuga--Satake Hodge conjecture has been proven and the cases in which the Lefschetz standard conjecture in degree two is known to hold. This way, we describe examples of applications of Theorem \ref{main thm} and Corollary \ref{main cor}.

\smallskip
As mentioned in the introduction, Hodge similarities between transcendental lattices of hyperk\"ahler manifolds appear naturally in two cases: as elements of totally real endomorphism fields of degree two, and as Hodge isomorphisms $T(Y)\rightarrow T(X)$, where $X$ and $Y$ are hyperk\"ahler manifolds with $T(Y)$ Hodge isometric to $T(X)(\lambda)$ for some $\lambda\in \Q_{>0}$.
\smallskip

Let us start from the case of hyperk\"ahler manifolds of generalized Kummer type. For these varieties the Kuga--Satake Hodge conjecture is proven in \cite{voisin2022footnotes} and the Lefschetz standard conjecture in degree two has been proven in \cite{foster2023lefschetz}.
We thus get our main families of examples of varieties satisfying the hypotheses of Corollary \ref{main cor}, and we conclude that every Hodge similarity between the transcendental lattices of two hyperk\"ahler manifolds of generalized Kummer type is algebraic. Note that the dimension of the transcendental lattice of a hyperk\"ahler manifold of generalized Kummer type is at most six-dimensional. Therefore, its endomorphism field is either a CM field or a totally real field of degree one or two. In all cases, we see that it is always generated by Hodge similarities. We therefore deduce the following:
\begin{thm}
\label{thm. similarities and gen kummer}
Let $X$ and $X'$ be hyperk\"ahler manifolds of generalized Kummer type such that $T(X)$ and $T(X')$ are Hodge similar. Then, every Hodge morphism between $T(X')$ and $T(X)$ is algebraic.
\end{thm}

\begin{rmk}
\label{rmk. Gen kummer of different dimension}
Taking $X=X'$ in Theorem \ref{thm. similarities and gen kummer}, we see that every Hodge morphism in $E\coloneqq\mathrm{End}_{\mathrm{Hdg}}(T(X))$ is algebraic.
Note that, Theorem \ref{thm. similarities and gen kummer} also covers the case where $X$ and $X'$ are hyperk\"ahler manifolds of generalized Kummer type with Hodge similar transcendental lattice but of different dimension. Let us briefly recall why this happens: recall that the second cohomology group of a hyperk\"ahler manifold $X$ of generalized Kummer type of dimension $2n$ satisfies
\[
(H^2(X,\Q),q_X)\simeq U^{\oplus 3}_\Q \oplus \Q\delta_n,
\]
where $(\delta_n)^2=-2(n+1)$.
Let $k$ be a positive integer. Using the fact that $U_\Q$ is isometric with $U_\Q(k)$, one sees that there is an isometry
\[
U^{\oplus 3}_\Q \oplus \Q\delta_{n'}\rightarrow (U^{\oplus 3}_\Q \oplus \Q\delta_n)(k),
\]
for $n'\coloneqq k(n+1)-1$.
In other words, there is a similarity
\[
\psi\colon U^{\oplus 3}_\Q \oplus \Q\delta_{n'}\rightarrow U^{\oplus 3}_\Q \oplus \Q\delta_n.
\]
Let $[\sigma']\in \mathbb{P}((U^{\oplus 3}_\Q \oplus\Q \delta_{n'})\otimes_\Q \C)$ be a class satisfying $(\sigma')^2=0$ and $(\sigma',\overline{\sigma'})>0$. Then, $[\sigma']$ determines a Hodge structure on the quadratic space $U^{\oplus 3}_\Q \oplus\Q\delta_{n'}$. Similarly the class $[\psi(\sigma')]$ satisfies the same hypotheses of $\sigma'$, hence, it defines a Hodge structure on $U^{\oplus 3}_\Q \oplus \Q\delta_{n}$. By the surjectivity of the period map, we obtain a hyperk\"ahler manifold $X'$ of $\mathrm{Kum}^{n'}$-type whose symplectic form is given by $\sigma'$ and a hyperk\"ahler manifold $X$ of $\mathrm{Kum}^{n}$-type whose symplectic form is given by $\psi(\sigma)$. By construction, the morphism $\psi$ defines a Hodge similarity between $T(X')$ and $T(X)$. Thus, $X$ and $X'$ satisfy the hypotheses for Theorem \ref{thm. similarities and gen kummer}, and we conclude that any Hodge morphism $T(X')\rightarrow T(X)$ is algebraic.
\end{rmk}

Let us briefly comment on the application of our result to the case of K3 surfaces. In this case, the Lefschetz standard conjecture is trivially true. Hence, applying Corollary \ref{main cor}, we get the following:
\begin{thm}
\label{thm. for K3 surfaces}
Let $S$ and $S'$ be K3 surfaces for which the Kuga--Satake Hodge conjecture holds. Then, every Hodge similarity between $T(S)$ and $T(S')$ is algebraic. \qed
\end{thm}
For K3 surfaces, the Kuga--Satake Hodge conjecture is in general not known. However, in \cite{floccari2022sixfolds}, the author proves it for the (countably many) four-dimensional families of K3 surfaces with transcendental lattice isometric to $T(K)(2)$ for a hyperk\"ahler manifold $K$ of generalized Kummer type of dimension six.
By \cite[Sec.\ 3]{van2008real}, there are one-dimensional subfamilies in these four-dimensional families of K3 surfaces which have totally real endomorphism field of degree two. As totally real fields of degree two are generated by Hodge similarities, Theorem \ref{thm. for K3 surfaces} proves the Hodge conjecture for the square of these K3 surfaces. As mentioned in the introduction, the Hodge conjecture for all powers of these particular K3 surfaces has been proven in \cite{varesco2022hodge} by extending the techniques introduced in \cite{schlickewei2010hodge}. The proof we provided here for the square of these K3 surfaces is however more direct since it does not involve the study of the Hodge conjecture for the Kuga--Satake varieties.

\smallskip
Finally, let us come to the case of hyperk\"ahler manifolds of $\mathrm{K3}^{[n]}$-type.
As mentioned in the introduction, the Lefschetz standard conjecture holds for these manifolds by \cite{charles2013standard}. Therefore, applying Corollary \ref{main cor}, we get the following:
\begin{thm}
\label{thm. for HK of Hilbert type}
Let $X$ and $X'$ be hyperk\"ahler manifolds of $\mathrm{K3}^{[n]}$- and $\mathrm{K3}^{[n']}$-type for which the Kuga--Satake Hodge conjecture holds. Then, every Hodge similarity between $T(X')$ and $T(X)$ is algebraic.\qed
\end{thm}
The Kuga--Satake Hodge conjecture has not been proven for hyperk\"ahler manifolds of $\mathrm{K3}^{[n]}$-type. In dimension six, this conjecture follows from the construction in \cite{floccari2022sixfolds} for the families of hyperk\"ahler manifolds of $\mathrm{K3}^{[3]}$-type which are resolution of the quotient of a hyperk\"ahler manifold of generalized Kummer type of dimension six by a symplectic group $G\simeq (\mathbb{Z}/2\mathbb{Z})^5$. This way, we obtain four-dimensional families of hyperk\"ahler manifolds of $\mathrm{K3}^{[3]}$-type which satisfy the hypotheses of Theorem \ref{thm. for HK of Hilbert type}.
\begin{rmk}
Note that Theorem \ref{main thm} and Corollary \ref{main cor} provide algebraic classes on $X'\times X$ whenever the Kuga--Satake correspondence is algebraic for $X$ and $X'$ and the transcendental lattices are Hodge similar. 
This also works when $X$ and $X'$ are not of the same deformation type, and when the dimensions of $X$ and $X'$ are not the same.
\end{rmk}

\bibliography{HyperK}
\bibliographystyle{plain}

\end{document}